\documentclass[11pt,reqno]{article}
\usepackage{latexsym}
\usepackage{amssymb,amsmath}
\usepackage{graphicx}
\input epsf
\makeatletter
\def\section{\@startsection {section}{1}{\z@}{-3.5ex plus -1ex minus
 -.2ex}{2.3ex plus .2ex}{\large\sc}}
\def\subsection{\@startsection{subsection}{2}{\z@}{-3.25ex plus -1ex minus
 -.2ex}{1.5ex plus .2ex}{\normalsize\sc}}
\makeatother
\makeatletter
\@addtoreset{equation}{section}

\makeatother 
\newcommand{\nc}{\newcommand}

\nc{\bea}{\begin{eqnarray}}
\nc{\eea}{\end{eqnarray}}
\nc{\be}{\begin{equation}}
\nc{\ee}{\begin{equation}}
\nc{\tr}{\mathop{\mbox{tr}}\nolimits}
\nc{\ad}{\mathop{\mbox{ad}}\nolimits}
\nc{\Tr}{\mathop{\mbox{Tr}}\nolimits}
\nc{\Det}{\mathop{\mbox{Det}}\nolimits}
\nc{\rk}{\mathop{\mbox{rk}}\nolimits}
\nc{\ra}{\rightarrow}
\nc{\Ra}{\Rightarrow}
\nc{\ot}{\otimes}
\nc{\non}{\nonumber\\}
\nc{\ZZ}{\mathbb{Z}}
\nc{\RR}{\mathbb{R}}

\newtheorem{theorem}{Theorem}
\newtheorem{proposition}{Proposition}
\newtheorem{lemma}{Lemma}

\newtheorem{conjecture}{Conjecture}

\newtheorem{question}{Question}
\newenvironment{proof}[1][Proof]{\begin{trivlist}
\item[\hskip\labelsep {\bfseries #1}]}{\end{trivlist}}

\def\1#1{^{(#1)}}
\def\la{\langle}
\def\ra{\rangle}

\begin{document}
\title{Numerical semigroups generated by squares, cubes and quartics\\ of
three consecutive integers}
\author{Leonid G. Fel\\ \\
Department of Civil Engineering, Technion, Haifa 32000, Israel}
\date{}
\maketitle
\vspace{-1cm}
\begin{abstract} 
We derive the polynomial representations for minimal relations of generating set
of numerical semigroups $R_n^k=\la(n-1)^k,n^k,(n+1)^k\ra$, $k=2,3,4$, $n\geq 3$.
We find also the polynomial representations for degrees of syzygies in the 
Hilbert series $H\left(z,R_n^k\right)$ of these semigroups, their Frobenius 
numbers $F\left(R_n^k\right)$ and genera $G\left(R_n^k\right)$. 
\end{abstract}
\section{Symmetric and nonsymmetric numerical semigroups $\langle(n-1)^k,n^k,
(n+1)^k\rangle$}\label{l1}
Numerical semigroups $S_3=\la d_1,d_2,d_3\ra$, generated by three integers, 
exhibit a non-trivial example of semigroups with well established relations 
\cite{fel06} between degrees of syzygies and values of generators. In this 
regards, the relations $f_k(d_1,d_2,d_3)=0$, imposed on generators, may increase
a number of semigroups with explicitly computable Hilbert series $H(z,S_3)$, 
Frobenius numbers $F(S_3)$ and genera $G(S_3)$. Usually the generators of such 
semigroups are represented as elements of some ordered sets: arithmetic 
\cite{bra42}, almost arithmetic \cite{rob56} or geometric \cite{op08}
progressions, Pythagorean triples \cite{kraf85,fel06}, Fibonacci \cite{mrr07,
f09} or Lucas \cite{f09} numbers and others.

Recently, the two 3-generated numerical semigroups were treated \cite{le15} to 
establish explicit expressions for their Frobenius numbers. These are semigroups
$R_n^2$ and $R_n^3$, generated by squares and cubes of three consecutive integers,
\bea
R_n^2=\la (n-1)^2,n^2,(n+1)^2\ra,\quad 
R_n^3=\la (n-1)^3,n^3,(n+1)^3\ra,\quad n\geq 3.\label{x1}
\eea
Using the Euclidean algorithm with negative \cite{ro78} and positive 
\cite{rr09} remainders for computing the Frobenius numbers, the authors 
\cite{le15} were able to find polynomial expressions in $n$ for $F\left(R_n^2
\right)$ and $F\left(R_n^3\right)$ on residue class of $n$ modulo 4 and 18, 
\bea
F_j\left(R_n^2\right)=\sum_{i=1}^3A^j_in^i,\;\;j=n\;(\bmod\;4),\qquad\;\; 
n\neq 3,4,5,6,9,13,\quad A^j_i\in{\mathbb Q},\label{x2}
\eea
\bea
F_j\left(R_n^3\right)=\sum_{i=1}^5B^j_in^i,\;\;j=n\;(\bmod\;18),\quad 
B^j_i\in{\mathbb Q},\quad n\neq 3,4,5,6,7,8,9,10,\qquad\label{x3}\\
\hspace{.5cm}
11,18,26,27,36,45,54,63,72,90,108,126,144,162,180,198,216,234,252,270.\nonumber
\eea

A long list of sophisticated formulas for $F_j\left(R_n^2\right)$ and $F_j\left(
R_n^3\right)$ in \cite{le15}, accompanied by 34 exclusions (6 cases for $R_n^2$ 
and 28 cases for $R_n^3$), poses a question to find another representation (Rep)
which allows incorporate all exclusions. Another reason to treat this problem 
again is to extend it on semigroups $R_n^4$ and find $H\left(z,R_n^4\right)$, 
$F\left(R_n^4\right)$ and $G\left(R_n^4\right)$, and to discuss how to deal with
a general case of $R_n^k$, $k\geq 5$.

Note that the excluding values of $n$ in (\ref{x2},\ref{x3}) give rise to the 
4 symmetric semigroups, $R_3^2$, $R_4^2$, $R_5^2$ and $R_3^3$. The rest of 30 
semigroups are nonsymmetric.
 
Simple considerations (Propositions \ref{pro1}, \ref{pro2}) show that no more 
symmetric numerical semigroups $R_n^2$ and $R_n^3$ do exist. To prove that we 
recall necessary conditions when a semigroup $\langle d_1,d_2,d_3\rangle$ 
becomes symmetric, 
\bea
&&a)\quad d_3\in\la d_1,d_2\ra,\quad\gcd(d_1,d_2)=1,\quad d_j>3,\qquad
\mbox{or}\label{x4}\\
&&b)\quad\la d_1,d_2,d_3\ra,\;d_j>3,\quad\mbox{satisfies Lemma \ref{lem1} 
adapted to three generators},\nonumber
\eea

\begin{lemma}\label{lem1}\cite{wata73}
Let a numerical semigroup $\la d_1,d_2,d_3\ra$ be given such that $d_1=a\delta
_1,d_2=a\delta_2$, $\delta_j\in{\mathbb Z}_+$ and $\gcd(\delta_1,\delta_2)=
\gcd(a,d_3)=1$. Then $\la d_1,d_2,d_3\ra$ is symmetric if and only if $d_3\in
\la\delta_1,\delta_2\ra$.
\end{lemma}
For semigroups, satisfying Lemma \ref{lem1}, the Frobenius number is given by 
\cite{john60,herz70},
\bea
F\left(\la d_1,d_2,d_3\ra\right)=aF\left(\la\delta_1,\delta_2\ra\right)+(a-1)
d_3,\quad F\left(\la\delta_1,\delta_2\ra\right)=\delta_1\delta_2-\delta_1-
\delta_2.\qquad\label{x5}
\eea
\subsection{Symmetric numerical semigroups $R_n^k$}\label{l11}  
Prove an exclusive property of semigroups $R_3^2$, $R_4^2$, $R_5^2$.

\begin{proposition}\label{pro1}
There exist only three symmetric semigroups $R_n^2$, $n=3,4,5$.
\end{proposition}
\begin{proof}
Note that semigroups $R_n^2$, $n=3,4$, are symmetric due to requirement 
(\ref{x4}a). Find more $n$ which satisfy (\ref{x4}a),
\bea
(n+1)^2=a_1(n-1)^2+a_2n^2,\quad a_1,a_2\in{\mathbb N},\quad n>4.\label{x6}
\eea
Simplifying the last equality we obtain the Diophantine Eq.
\bea
(a_1+a_2-1)(n-4)^2+2(3a_1+4a_2-5)(n-4)+9a_1+16a_2-25=0,\nonumber
\eea
with constraints (\ref{x6}) on three variables $a_1,a_2,n$ which has no 
solutions.

Consider another way to symmetrize $R_n^2$ by providing condition (\ref{x4}b), 
which may occur only when $n=2p+1$ and results in the Diophantine Eq. in $b_1,
b_2,p$,
\bea
(2p+1)^2=b_1p^2+b_2(p+1)^2,\quad b_1,b_2\in{\mathbb N},\quad p\geq 2,\label{x7}
\eea
Solving (\ref{x7}) as a quadratic equation in $p$, we get
\bea
p=\frac{2-b_2\pm\Theta}{b_1+b_2-4},\quad \Theta^2=b_1+b_2-b_1b_2,\nonumber
\eea
Combining the last expression with constraints in (\ref{x7}) we find only one 
appropriate solution, $b_1=4$, $b_2=1$, $p=2$, which gives rise to semigroup
$R_5^2$.
\end{proof}
The next Propositions deal with symmetric semigroups $R_n^k$, $k=3,4$.
\begin{proposition}\label{pro2}
There exists only one symmetric semigroup $R_n^3$, $n=3$.
\end{proposition}
\begin{proposition}\label{pro3}
There exist only three symmetric semigroups $R_n^4$, $n=3,5,7$.
\end{proposition}
Their proofs are similar to proof of Proposition \ref{pro1} but much more 
cumbersome and therefore are given in \ref{app1}.
\subsection{Nonsymmetric numerical semigroups generated by three integers}
\label{l12}
A brief analysis in the previous section focuses us on nonsymmetric semigroups 
only. In the present paper we calculate the Hilbert series for such semigroups 
$R_n^k$, $k=2,3,4$, making use of approach of minimal relations for three 
generators $d_1,d_2,d_3$. Recall this approach following author's article 
\cite{fel06}.

Let a nonsymmetric numerical semigroup $S_3=\la d_1,d_2,d_3\ra$, $\gcd(d_1,d_2,
d_3)=1$, $d_j\geq 3$, be given by matrix of minimal relations, ${\cal A}_3$,
where $a_{ij}\in{\mathbb Z}_+$,
\bea
{\cal A}_3\left(\begin{array}{c}d_1\\d_2\\d_3\end{array}\right)\!=\!
\left(\begin{array}{r}0\\0\\0\end{array}\right)\!,\;\;{\cal A}_3=\!\left(
\begin{array}{rrr}a_{11} & -a_{12} & -a_{13}\\-a_{21} & a_{22} & -a_{23}\\
-a_{31} & -a_{32} & a_{33} \end{array}\right),\;\;
\left\{\begin{array}{r}\gcd(a_{11},a_{12},a_{13})=1\\ \gcd(a_{21},a_{22},
a_{23})=1\\\gcd(a_{31},a_{32},a_{33})=1\end{array}\right.,\quad\label{x8}
\eea
\bea
a_{11}&=&\min\left\{v_{11}\;|\;v_{11}\geq 2,\;v_{11}d_1=v_{12}d_2+v_{13}d_3,
\;v_{12},v_{13}\in {\mathbb N}\cup\{0\}\right\}\;,\nonumber\\
a_{22}&=&\min\left\{v_{22}\;|\;v_{22}\geq 2,\;v_{22}d_2=v_{21}d_1+v_{23}d_3,
\;v_{21},v_{23}\in {\mathbb N}\cup\{0\}\right\}\;,\label{x9}\\
a_{33}&=&\min\left\{v_{33}\;|\;v_{33}\geq 2,\;v_{33}d_3=v_{31}d_1+v_{32}d_2,
\;v_{31},v_{32}\in {\mathbb N}\cup\{0\}\right\}\;.\nonumber
\eea
All matrix elements $a_{ij}$ are non-negative integers \cite{john60,fel06} such 
that
\bea
&&a_{11}=a_{21}+a_{31},\quad a_{22}=a_{12}+a_{32},\quad a_{33}=a_{13}+a_{23},
\quad a_{jj}\geq 2,\quad a_{ij}\geq 1,\;i\neq j,\nonumber\\
&&d_1=a_{22}a_{33}-a_{23}a_{32},\qquad d_2=a_{33}a_{11}-a_{31}a_{13},\qquad 
d_3=a_{11}a_{22}-a_{12}a_{21}.\label{x10}
\eea
Then the rational Rep of the Hilbert series $H(z,S_3)$, the Frobenius number 
$F(S_3)$ and genus $G(S_3)$ are given by formulas \cite{fel06},
\bea
&&H(z,S_3)=\left(1-z^{a_{11}d_1}-z^{a_{22}d_2}-z^{a_{33}d_3}+z^{b_{11}}+z^{b_
{22}}\right)\prod_{i=1}^3\left(1-z^{d_i}\right)^{-1},\nonumber\\
&&b_{11}=D_0+D_1,\quad b_{22}=D_0+D_2,\quad a_{11}d_1+a_{22}d_2+a_{33}d_3=
b_{11}+b_{22},\nonumber\\
&&F(S_3)=\max\left\{F_1,F_2\right\},\quad 2G(S_3)=1+D_0+D_1+D_2-D_3,\label{x11}
\\
&&F_1=b_{11}-D_3,\qquad F_2=b_{22}-D_3.\nonumber\\
&&D_0=a_{11}a_{22}a_{33},\quad D_1=a_{12}a_{23}a_{31},\quad D_2=a_{13}a_{32}
a_{21},\quad D_3=d_1+d_2+d_3.\nonumber
\eea
Based on (\ref{x11}) we reduce a large number of exclusive semigroups 
\cite{le15} with $F(S_3)$ which differ from polynomials (\ref{x2}). This 
exclusion may happen when in different ranges of $n$ a difference $D_1(n)-D_2(n
)$ may change its sign. In other words, the both sequences $F_1(n)$ and $F_2(n)$
contribute to the polynomial Rep of $F(S_3)$.
\section{Numerical semigroups $R_n^2$, $n\geq 6$}\label{l2}
Write the third relation in (\ref{x9}) for $R_n^2$, $a_{33}(n+1)^2=a_{32}n^2+
a_{31}(n-1)^2$, i.e.,
\bea
a_{32}n^2=(a_{33}-a_{31})\left(n^2+1\right)+2(a_{33}+a_{31})n.\nonumber
\eea
Choose $a_{33}=a_{31}$ that results in $a_{32}=4a_{33}/n$. The whole matrix 
${\cal A}_3$ satisfies relations (\ref{x8}),
\bea
\left(\begin{array}{ccc}a_{21}+a_{33}\;&\;4a_{33}/n-a_{22}\;&\;a_{23}-a_{33}\\
-a_{21}\;&\;a_{22}\;&\;-a_{23}\\-a_{33}\;&\;-4a_{33}/n\;&\;a_{33}\end{array}
\right),\quad\left.\begin{array}{rcl}(n-1)^2&=&a_{33}\left(a_{22}-4a_{23}/n
\right),\\ n^2&=&a_{33}(a_{21}+a_{23}),\\(n+1)^2&=&a_{33}\left(a_{22}+4a_{21}/n
\right).\end{array}\right.\nonumber
\eea
In order to provide all entries in ${\cal A}_3$ be integers we consider four
different cases, $n=j\;(\bmod\;4)$.

{\bf 1}. $n=4m$, $a_{33}=a_{31}=m$, $a_{32}=1$.
\bea
\left(\begin{array}{ccc}a_{21}+m\;&\;1-a_{22}\;&\;a_{23}-m\\-a_{21}\;&\;a_{22}
\;&\;-a_{23}\\-m\;&\;-1\;&\;m\end{array}\right),\quad\left.\begin{array}{rcl}
(4m-1)^2&=&a_{22}m-a_{23},\\(4m+1)^2&=&a_{22}m+a_{21}.\end{array}\right.
\label{y1}
\eea
Two Eqs. (\ref{y1}) for $a_{21}$, $a_{22}$, $a_{23}$ allow to choose the 
following parameterization,
\bea
a_{21}=km+1,\quad a_{22}=16m+8-k,\quad a_{23}=(16-k)m-1,\quad 1\leq k\leq 15,
\nonumber
\eea
\bea
\left(\begin{array}{ccc}(k+1)m+1\;&\;-(16m+7-k)\;&\;-[(k-15)m+1]\\-(km+1)\;&\;
16m+8-k\;&\;-[(16-k)m-1]\\-m\;&\;-1\;&\;m\end{array}\right).\nonumber
\eea
Since $a_{13},a_{23}\geq 1$ there is only one solution $k=15$, $m\geq 2$, that 
gives
\bea
\left(\begin{array}{ccc}16m+1\;&\;-8(2m-1)\;&\;-1\\-(15m+1)\;&\;16m-7\;&\;
-(m-1)\\-m\;&\;-1\;&\;m\end{array}\right),\left.\begin{array}{rcl}G&=&
4m(34m^2-21m+2),\\F&=&20,\quad m=1,\\F&=&272m^3-168m^2+m-2,\;\;m\geq 2.
\end{array}\right.\nonumber
\eea
For the rest semigroups we skip the parameterization details of $a_{2j}$ and 
give the final formulas.

{\bf 2}. $n=4m+2$, $a_{33}=a_{31}=2m+1$, $a_{32}=2$.
\bea
\left(\begin{array}{ccc}9m+5&-(8m-1)&-(m+1)\\-(7m+4)&8m+1&-m\\-(2m+1)&-2&2m+1
\end{array}\right),\left.\begin{array}{rcl}G&=&m(80m^2+71m+16),\\
F&=&312,\quad m=1,\\F&=&160m^3+128m^2+10m-9,\;\;m\geq 2.\end{array}\right.
\nonumber
\eea

{\bf 3}. $n=4m+1$, $a_{33}=a_{31}=4m+1$, $a_{32}=4$.
\bea
\left(\begin{array}{ccc}7m+2\;&\;-4(m-1)\;&\;-(3m+1)\\-(3m+1)\;&\;4m\;&\;-m\\
-(4m+1)\;&\;-4\;&\;4m+1\end{array}\right),\left.\begin{array}{rcl}G&=&2m
(32m^2+9m+1),\\F&=&128m^3-20m-5,\;\;m\geq 4\\F&=&112m^3+48m^2+8m-1,\;\;m\leq 3.
\end{array}\right.\nonumber
\eea

{\bf 4}. $n=4m+3$, $a_{33}=a_{31}=4m+3$, $a_{32}=4$.   
\bea
\left(\begin{array}{ccc}5m+4\;&\;-4m\;&\;-(m+1)\\-(m+1)\;&\;4(m+1)\;&\;-(3m+2)\\
-(4m+3)\;&\;-4\;&\;4m+3\end{array}\right),\left.\begin{array}{rcl}
G&=&2(32Nm^3+57m^2+33m+6),\\\\F&=&128m^3+2242+124m+19,\;\;m\geq 1.\end{array}
\right.\nonumber
\eea
The above formulas for $F(R_n^2)$ coincide with those obtained in \cite{le15} 
when $b_{22}>b_{11}$. In the opposite case ($b_{22}<b_{11}$) we arrive at the 
other formulas, e.g., as in the case $n=4m+1$, $m\leq 3$. There exist 3 
exceptional symmetric semigroups $R_3^2$, $R_4^2$, $R_5^2$, with minimal 
relations which do not obey the matrix Reps presented above.
\section{Numerical semigroups $R_n^3$, $n\geq 4$}\label{l3}
Write the third relation in (\ref{x9}) for $R_n^3$,
\bea
a_{32}n^3=(a_{33}-a_{31})\left(n^3+3n\right)+(a_{33}+a_{31})(3n^2+1).\label{b1}
\eea
Choose the Rep $a_{31}=(pn+q)$, $a_{33}=(pn-q)$, $p,q\in{\mathbb Q}$, and 
insert it into (\ref{b1}),
\bea
a_{32}n^2=2p(3n^2+1)-2q(n^2+3),\quad\mbox{or}\quad 
a_{32}=6p-2q+2(p-3q)/n^2.\nonumber
\eea
To eliminate the dependence of $a_{32}$ on $n^{-2}$ in the last relation we put
\bea
p=3q,\quad a_{31}=q(3n+1),\quad a_{32}=16q,\quad a_{33}=q(3n-1).\nonumber
\eea
To satisfy $\gcd(a_{31},a_{32},a_{33})=1$ in (\ref{x9}), we have to 
distinguish two different cases: $q=1$ if $n=2N$ and $q=1/2$ if $n=2N+1$.
\subsection{Numerical semigroups $R_n^3$, $\;n=0\;(\bmod\;2)$}\label{l31}
The matrix of minimal relations ${\cal A}_3$ reads,
\bea
\left(\begin{array}{ccc}a_{21}+6N+1\;&\;16-a_{22}\;&\;a_{23}-(6N-1)\\-a_{21}
\;&\;a_{22}\;&\;-a_{23}\\-(6N+1)\;&\;-16\;&\;(6N-1)\end{array}\right),\!\!\left.\begin{array}{rcl}
(2N-1)^3\!\!&\!=\!&\!\!a_{22}(6N-1)-16a_{23},\\ \\(2N+1)^3\!\!&\!=\!&\!\!a_{22}
(6N+1)+16a_{21}.\end{array}\right.\quad\label{b2}
\eea
Two Eqs. (\ref{b2}) need that at least two of $a_{2j}$ be polynomials in 
$N$ 
of the 2nd degree, e.g., $a_{21}$ and $a_{22}$ are quadratic polynomials. To 
balance the cubic degrees in (\ref{b2}) choose $a_{2j}$ as polynomials on 
residue class of $N$ modulo $t_3$ which will be found later, i.e., $N=t_3m+j$,
$0\leq j<t_3$,
\bea
a_{21}=r_2m^2+r_1m+r_0,\quad a_{22}=k_2m^2+k_1m+k_0,\quad a_{23}=l_1m+l_0.
\label{b3}
\eea
Substitute (\ref{b3}) into (\ref{b2}) and obtain
\bea
&&k_2=r_2=\frac{4t_3^2}{3},\quad k_1=\frac{8t_3}{9}(3j-2),\quad 
3t_3k_0+4(r_1-l_1)=\frac{t_3}{3}\left(9+16j+12j^2\right),\nonumber\\
&&(6j+1)k_0+16r_0=(2j+1)^3,\quad (6j-1)k_0-16l_0=(2j-1)^3,\quad
\frac{r_1+l_1}{1+12j}=\frac{2t_3}{9}.\nonumber
\eea
The minimal value of $t_3$, providing $k_1\in{\mathbb Z}_+$ in the above 
equalities be integer, is $t_3=9$, that leads to $k_1=8(3j-2)$ and $k_2=r_2=
108$. The other five parameters, $r_0(j),r_1(j),k_0(j),l_0(j)$ and $l_1(j)$ may 
be found if we find numerically matrix ${\cal A}_3$ for nine first semigroups 
$R_{18+2j}^3$, $0\leq j\leq 8$.

{\bf 1}. $n=18m$, $m\geq 1$,
\bea
\left(\begin{array}{ccc}108m^2+55m+1\;&\;-(108m^2-16m-15)\;&\;-(53m-1)\\
-m(108m+1)\;&\;108m^2-16m+1\;&\;-m\\-(54m+1)\;&\;-16\;&\;54m-1\end{array}
\right),\nonumber
\eea
\bea
&&G=314928m^5+110808m^4+16632m^3+532m^2-62m,\nonumber\\
&&F=629856 m^5+215784m^4+34020m^3+1890m^2-109m-1,\quad m\leq 15,\nonumber\\
&&F=629856 m^5+221616m^4-58320m^3+1944m^2-108m-1,\quad m\geq 16.\nonumber
\eea

{\bf 2}. $n=18m+2$, $m\geq 1$; $\qquad n\neq 2$
\bea
\left(\begin{array}{ccc}108m^2+37m+3\;&\;-(108m^2+8m-3)\;&\;-(11m+1)\\
-(108m^2-17m-4)\;&\;108m^2+8m+13\;&\;-(43m+4)\\-(54m+7)\;&\;-16\;&\;54m+5
\end{array}\right),\nonumber
\eea
\bea
&&G=314928m^5+285768m^4+104760 m^3+15244m^2+786m+6,\nonumber\\
&&F=629856m^5+571536m^4+190512m^3+31752m^2+2548m+75.\nonumber
\eea

{\bf 3}. $n=18m+4$, $m\geq 1$; $\qquad n\neq 4$.
\bea
\left(\begin{array}{ccc}108m^2+55m+7\;&\;-(108m^2+32m+1)\;&\;-(5m+1)\\
-(108m^2+m-6)\;&\;108m^2+32m+17\;&\;-(49m+10)\\-(54m+13)\;&\;-16\;&\;54m+11
\end{array}\right),\nonumber
\eea
\bea
&&G=314928m^5+460728m^4+270648m^3+77476m^2+10674 m+564,\nonumber\\
&&F=629856 m^5+921456m^4+532656m^3+153144m^2+21812m+1223.\nonumber
\eea

{\bf 4}. $n=18m+6$, $m\geq 1$; $\qquad n\neq 6$.
\bea
\left(\begin{array}{ccc}108m^2+109m+25\;&\;-(108m^2+56m-3)\;&\;-(35m+11)\\
-(108m^2+55m+6)\;&\;108m^2+56m+13\;&\;-(19m+6)\\-(54m+19)\;&\;-16\;&\;54m+17
\end{array}
\right),\nonumber
\eea
\bea
&&G=314928m^5+635688m^4+514296m^3+202780m^2+38434m+2778,\nonumber\\
&&F=629856m^5+1271376m^4+968112m^3+355752m^2+63828m+4499.\nonumber 
\eea

{\bf 5}. $n=18m+8$, $m\geq 0$,
\bea
\left(\begin{array}{ccc}108m^2+145m+44\;&\;-(108m^2+80m+1)\;&\;-(47m+20)\\
-(108m^2+91m+19)\;&\;108m^2+80m+17\;&\;-(7m+3)\\-(54m+25)\;&\;-16\;&\;54m+23
\end{array}\right),\nonumber
\eea
\bea
&&G=314928m^5+810648m^4+835704m^3+428308m^2+108658m+10888,\nonumber\\
&&F_{m=0,1}=629856m^5+1580472m^4+1604772m^3+821178m^2+210979m+21700,\nonumber\\
&&F_{m\geq 2}=629856m^5+1621296 m^4+1590192 m^3+753624 m^2+173908m+15695.
\nonumber
\eea

{\bf 6}. $n=18m+10$, $m\geq 0$,
\bea
\left(\begin{array}{ccc}108m^2+163m+58\;&\;-(108m^2+104m+13)\;&\;-(41m+22)\\
-(108m^2+109m+27)\;&\;108m^2+104m+29\;&\;-(13m+7)\\-(54m+31)\;&\;-16\;&\;54m+29
\end{array}\right),\nonumber
\eea
\bea
&&G=314928m^5+985608m^4+1234872m^3+769612m^2+237666m+29022,\nonumber\\
&&F=55222,\quad m=0,\nonumber\\
&&F_{m\geq 1}=629856m^5+1971216 m^4+2398896 m^3+1429704 m^2+419252m+48539.
\nonumber
\eea

{\bf 7}. $n=18m+12$, $m\geq 0$,
\bea
\left(\begin{array}{ccc}108m^2+163m+61\;&\;-(108m^2+128m+33)\;&\;-(17m+11)\\
-(108m^2+109m+24)\;&\;108m^2+128m+49\;&\;-(37m+24)\\-(54m+37)\;&\;-16\;&\;54m+35
\end{array}\right),\nonumber
\eea
\bea
&&G=314928m^5+1160568 m^4+1711800m^3+1257796 m^2+459058m+66444,\nonumber\\
&&F=629856m^5+2321136 m^4+3394224 m^3+2466936 m^2+892404m+128663.\nonumber
\eea

{\bf 8}. $n=18m+14$, $m\geq 0$,
\bea
\left(\begin{array}{ccc}108m^2+199m+90\;&\;-(108m^2+152m+45)\;&\;-(29m+22)\\
-(108m^2+145m+47)\;&\;108m^2+152m+61\;&\;-(25m+19)\\-(54m+43)\;&\;-16\;&\;54m+41
\end{array}\right),\nonumber
\eea
\bea
&&G=314928m^5+1335528 m^4+2266488 m^3+1917916 m^2+807378 m+135042,\nonumber\\
&&F=629856m^5+2671056 m^4+4482864 m^3+3730536m^2+1541908m+253539.\nonumber
\eea

{\bf 9}. $n=18m+16$, $m\geq 0$,
\bea
\left(\begin{array}{ccc}108m^2+217m+108\;&\;-(108m^2+176m+65)\;&\;-(23m+20)\\
-(108m^2+163m+59)\;&\;108m^2+176m+81\;&\;-(31m+27)\\-(54m+49)\;&\;-16\;&\;54m+47
\end{array}\right),\nonumber
\eea
\bea
&&G=314928m^5+1510488 m^4+2898936 m^3+2776756 m^2+1325266 m+251824,\nonumber\\
&&F=629856m^5+3020976 m^4+5758128m^3+5458968m^2+2576660m+484767.\nonumber
\eea
\subsection{Numerical semigroups $R_n^3$, $\;n=1\;(\bmod\;2)$}\label{l32}

The matrix of minimal relations ${\cal A}_3$ reads,
\bea
\left(\begin{array}{ccc}a_{21}+(3N+2)\;&\;8-a_{22}\;&\;a_{23}-(3N+1)\\
-a_{21}\;&\;a_{22}\;&\;-a_{23}\\-(3N+2)\;&\;-8\;&\;3N+1\end{array}\right),
\left.\begin{array}{rcl}(2N)^3\!&\!=\!&\!a_{22}(3N+1)-8a_{23},\\ \\(2N+2)^3\!&\!
=\!&\!a_{22}(3N+2)+8a_{21}.\end{array}\right.\nonumber
\eea
We skip intermediate calculations repeating the procedure performed in section 
\ref{l31}.

{\bf 1}. $n=18m+1$, $m\geq 0$,
\bea
\left(\begin{array}{ccc}216m^2+29m+1\;&\;-(216m^2-8m)\;&\;-m\\
-(216m^2+2m-1)\;&\;216m^2-8m+8\;&\;-(26m+1)\\-(27m+2)\;&\;-8\;&\;27m+1
\end{array}\right),\nonumber
\eea
\bea
&&G=629856m^5+160380 m^4+23220 m^3+2330 m^2+73 m,\nonumber\\
&&F=1259712m^5+320760m^4+44712m^3+4644m^2+154m-1.\nonumber
\eea

{\bf 2}. $n=18m+3$, $m\geq 1$; $\qquad n\neq 3$.
\bea
\left(\begin{array}{ccc}216m^2+83m+8\;&\;-(216m^2+40m)\;&\;-(7m+1)\\
-(216m^2+56m+3)\;&\;216m^2+40m+8\;&\;-(20m+3)\\-(27m+5)\;&\;-8\;&\;27m+4
\end{array}\right),\nonumber
\eea
\bea
&&G=629856m^5+510300 m^4+172260 m^3+30134 m^2+2657 m+91,\nonumber\\
&&F=181,\quad m=0,\nonumber\\
&&F=1259712m^5+1020600m^4+332424m^3+55404m^2+4698m+157,\quad m\geq 1.\nonumber
\eea

{\bf 3}. $n=18m+5$, $m\geq 0$,
\bea
\left(\begin{array}{ccc}216m^2+128m+19\;&\;-(216m^2+88m+8)\;&\;-(4m+1)\\
-(216m^2+101m+11)\;&\;216m^2+88m+16\;&\;-(23m+6)\\-(27m+8)\;&\;-8\;&\;27m+7
\end{array}\right),\nonumber
\eea
\bea
&&G=629856m^5+860220 m^4+476820 m^3+134186 m^2+18993 m+1066,\nonumber\\
&&F=1259712m^5+1720440m^4+946728m^3+263412m^2+37082m+2107.\nonumber
\eea

{\bf 4}. $n=18m+7$, $m\geq 0$,
\bea
\left(\begin{array}{ccc}216m^2+191m+42\;&\;-(216m^2+136m+16)\;&\;-(19m+7)\\
-(216m^2+164m+31)\;&\;216m^2+136m+24\;&\;-(8m+3)\\-(27m+11)\;&\;-8\;&\;27m+10
\end{array}\right),\nonumber
\eea
\bea
&&G=629856m^5+1210140 m^4+936900 m^3+365246 m^2+71609 m+5637,\nonumber\\
&&F=10745,\quad m=0,\nonumber\\
&&F_{m\geq 1}=1259712m^5+2420280m^4+1840968m^3+693468m^2+129322m+9537.\nonumber
\eea

{\bf 5}. $n=18m+9$, $m\geq 0$,
\bea
\left(\begin{array}{ccc}216m^2+245m+69\;&\;-(216m^2+184m+32)\;&\;-(25m+12)\\
-(216m^2+218m+55)\;&\;216m^2+184m+40\;&\;-(2m+1)\\-(27m+14)\;&\;-8\;&\;27m+13
\end{array}\right),\nonumber
\eea
\bea
&&G=629856m^5+1560060 m^4+1552500 m^3+776234 m^2+195001 m+19684,\nonumber\\
&&F_{m\leq 3}=1259712m^5+3108456m^4+3083184m^3+1537596m^2+385666m+38919,
\nonumber\\
&&F_{m\geq 4}=1259712m^5+3120120m^4+3061800m^3+1488132m^2+358074m+34087.
\nonumber
\eea

{\bf 6}. $n=18m+11$, $m\geq 0$,
\bea
\left(\begin{array}{ccc}216m^2+290m+97\;&\;-(216m^2+232m+56)\;&\;-(22m+13)\\
-(216m^2+263m+80)\;&\;216m^2+232m+64\;&\;-(5m+3)\\-(27m+17)\;&\;-8\;&\;27m+16
\end{array}\right),\nonumber
\eea
\bea
&&G=629856m^5+1909980 m^4+2323620 m^3+1417694 m^2+433745 m+53223,\nonumber\\
&&F=103589,\quad m=0,\nonumber\\
&&F_{m\geq 1}=1259712m^5+3819960m^4+4609224m^3+2766636m^2+826058m+98125.
\nonumber
\eea

{\bf 7}. $n=18m+13$, $m\geq 0$,
\bea
\left(\begin{array}{ccc}216m^2+326m+123\;&\;-(216m^2+280m+90)\;&\;-(10m+7)\\
-(216m^2+299m+103)\;&\;216m^2+280m+96\;&\;-(17m+12)\\-(27m+20)\;&\;-8\;&\;27m+19
\end{array}\right),\nonumber
\eea
\bea
&&G=629856m^5+2259900 m^4+3250260 m^3+2342114 m^2+845465m+122286,
\nonumber\\
&&F=1259712m^5+3120120m^4+3061800m^3+1488132m^2+358074m+34087.\nonumber
\eea

{\bf 8}. $n=18m+15$, $m\geq 0$,
\bea
\left(\begin{array}{ccc}216m^2+380m+167\;&\;-(216m^2+328m+120)\;&\;-(16m+13)\\
-(216m^2+353m+144)\;&\;216m^2+328m+128\;&\;-(11m+9)\\-(27m+23)\;&\;-8\;&\;27m+22
\end{array}\right),\nonumber
\eea
\bea
&&G=629856m^5+2609820 m^4+4332420 m^3+3601550 m^2+1499153 m+249937,\nonumber\\
&&F=1259712m^5+5219640 m^4+8637192 m^3+7135452 m^2+2943162 m+484897.
\nonumber
\eea

{\bf 9}. $n=18m+17$, $m\geq 0$,
\bea
\left(\begin{array}{ccc}216m^2+425m+209\;&\;-(216m^2+376m+160)\;&\;-(13m+12)\\
-(216m^2+398m+183)\;&\;216m^2+376m+168\;&\;-(14m+13)\\-(27m+26)\;&\;-8\;&\;
27m+25\end{array}\right),\nonumber
\eea
\bea
&&G=629856m^5+2959740 m^4+5570100 m^3+5247626 m^2+2474713 m+467304,\nonumber\\
&&F=1259712m^5+5919480 m^4+11117736 m^3+10433124 m^2+4892186 m+917039.\nonumber
\eea
Formulas for $F(R_n^3)$ in this section coincide with those obtained in 
\cite{le15} if $b_{22}>b_{11}$.
\subsection{Exceptional semigroups $R_n^3$, $n=4,6$}\label{l33}
There exist 3 exceptional semigroups (1 symmetric $R_3^3\equiv\la 2^3,3^3\ra$ 
and 2 nonsymmetric $R_4^3$, $R_6^3$) with minimal relations which do not obey 
the matrix Reps presented in section \ref{l31}, \ref{l32}. Below we give Reps 
of two nonsymmetric semigroups.
\bea
R_4^3:\left(\begin{array}{rrr}7&-1&-1\\-1&18&-9\\-6&-17&10\end{array}\right),
\quad\left.\begin{array}{l}G=558\\F=1098\end{array}\right.,\qquad
R_6^3:\left(\begin{array}{rrr}31&-10&-5\\-6&13&-6\\-25&-3&11\end{array}\right),
\quad\left.\begin{array}{l}G=2670\\F=5249\end{array}\right..\nonumber
\eea
\section{Numerical semigroups $R_n^4$, $n\geq 4$}\label{l4}
These semigroups were not studied in \cite{le15}, however using a weak 
argumentation its authors predict that $F(R_n^4)$ is given by polynomial 
expressions in $n$ on residue class of $n$ modulo 88 {\em 'whereas experimental 
tests make us believe that we need 40 formulas'} \cite{le15}.

Assuming that the matrix of minimal relations comprise the polynomial 
expressions in $n$ on residue class of $n$ modulo 40, we have calculated 
numerically these matrices in two different cases $n=0,1\;(\bmod\;2)$. In 
section \ref{l44} we give also expressions for genus of semigroups $R_{20m+9}^4$
and $R_{20m-9}^4$ which illustrate the forthcoming Theorem \ref{the2}.
\subsection{Numerical semigroups $R_n^4$, $\;n=0\;(\bmod\;2)$}\label{l41}
{\bf 1}. $n=40m$, $m\geq 4$; $\qquad n\neq 40,80,120$.
\bea
\left(\begin{array}{rrr}8160m^2+161m+1\;&\;-(320m^2-1280m+1)\;&\;
-(7840m^2-159m+1)\\-(160m^2+m)&320m^2+1&-(160m^2-m)\\-(8000m^2+160m+1)&-1280m&8000m^2-160m+1
\end{array}\right)\nonumber
\eea
{\bf 2}. $n=40m+2$, $m\geq 1$; $\qquad n\neq 2$.
\bea
\left(\begin{array}{rrr}4640m^2+526m+15\;&\;-(4480m^2+64m+6)\;&\;
-(160m^2-18m-1)\\-(4400m^2+495m+14)\;&\;4800m^2+160m+11\;&\;-(400m^2+65m+2)\\
-(240m^2+31m+1)&-(320m^2+96m+5)&560m^2+47m+1\end{array}\right)\nonumber
\eea
{\bf 3}. $n=40m+4$, $m\geq 0$.
\bea
\left(\begin{array}{rrr}2880m^2+616m+33\;&\;-8(320m^2+32m+1)\;&\;
-(320m^2+40m+1)\\-(2240m^2+473m+25)\;&\;2880m^2+448m+25 \;&\;-(640m^2+135m+7)\\
-(640m^2+143m+8)\;&\;-(320m^2+192m+17)\;&\;960m^2+175m+8\end{array}\right)
\nonumber
\eea
{\bf 4}. $n=40m+6$, $m\geq 1$; $\qquad n\neq 6$.
\bea
\left(\begin{array}{rrr}2800m^2+885m+70\;&\;-(1600m^2+160m-7)\;&\;
-(1200m^2+325m+22)\\-(1760m^2+550m+43)\;&\;1920m^2+448m+30\;&\;-(160m^2+58m+5)\\
-(1040m^2+335m+27)\;&\;-(320m^2+288m+37)\;&\;1360m^2+383m+27\end{array}\right)
\nonumber
\eea
{\bf 5}. $n=40m+8$, $m\geq 0$.
\bea
\left(\begin{array}{rrr}2240m^2+932m+97\;&\;-4(320m^2+64m+1)\;&\;
-(960m^2+356m+33)\\-(800m^2+325m+33)\;&\;1600m^2+640m+69\;&\;-(800m^2+315m+31)
\\-(1440m^2+607m+64)\;&\;-(320m^2+384m+65)\;&\;1760m^2+671m+64\end{array}
\right)\nonumber
\eea
{\bf 6}. $n=40m+10$, $m\geq 1$; $\qquad n\neq 10$.
\bea
\left(\begin{array}{rrr}2480m^2+1283m+166\;&\;-(960m^2+160m-17)\;&\;
-(1520m^2+723m+86)\\-(640m^2+324m+41)\;&\;1280m^2+640m+84\;&\;-(640m^2+316m+39)
\\-(1840m^2+959m+125)\;&\;-(320m^2+480m+101)\;&\;2160m^2+1039m+125\end{array}
\right)\nonumber
\eea
{\bf 7}. $n=40m+12$, $m\geq 0$.
\bea
\left(\begin{array}{rrr}1280m^2+787m+121\;&\;-(960m^2+448m+51)\;&\;
-(320m^2+179m+25)\\-(320m^2+183m+26)\;&\;2240m^2+1472m+247\;&\;
-(1920m^2+1129m+166)\\-(960m^2+604m+95)\;&\;-(1280m^2+1024m+196)\;&\;
2240m^2+1308m+191\end{array}\right)\nonumber
\eea
{\bf 8}. $n=40m+14$, $m\geq 1$; $\qquad n\neq 14$.
\bea
\left(\begin{array}{rrr}2720m^2+1954m+351\;&\;-(640m^2+64m-54)\;&\;
-(2080m^2+1410m+239)\\-(80m^2+51m+8)\;&\;960m^2+736m+143\;&\;-(880m^2+605m+104)
\\-(2640m^2+1903m+343)\;&\;-(320m^2+672m+197)\;&\;2960m^2+2015m+343\end{array}
\right)\nonumber
\eea
{\bf 9}. $n=40m+16$, $m\geq 0$.
\bea
\left(\begin{array}{rrr}1920m^2+1570m+321\;&\;-(640m^2+256m+2)\;&\;
-(1280m^2+994m+193)\\-(800m^2+645m+130)\;&\;1600m^2+1280m+261\;&\;
-(800m^2+635m+126)\\-(1120m^2+925m+191)\;&\;-(960m^2+1024m+259)\;&\;
2080m^2+1629m+319\end{array}\right)\nonumber
\eea
{\bf 10}. $n=40m+18$, $m\geq 0$.
\bea
\left(\begin{array}{rrr}1120m^2+1026m+235\;&\;-(640m^2+448m+74)\;&\;
-(480m^2+418m+91)\\-(1040m^2+937m+211)\;&\;2880m^2+2656m+621\;&\;
-(1840m^2+1639m+365)\\-(80m^2+89m+24)\;&\;-(2240m^2+2208m+547)\;&\;
2320m^2+2057m+456\end{array}\right)\nonumber
\eea
{\bf 11}. $n=40m+20$, $m\geq 2$; $\qquad n\neq 60$.
\bea
\left(\begin{array}{rrr}4160m^2+4241m+1081\;&\;-(320m^2-320m-239)\;&\;
-(3840m^2+3761m+921)\\-(320m^2+322m+81)\;&\;640m^2+640m+162\;&\;
-(320m^2+318m+79)\\-(3840m^2+3919m+1000)\;&\;-(320m^2+960m+401)\;&\;
4160m^2+4079m+1000\end{array}\right)\nonumber
\eea
{\bf 12}. $n=40m-18$, $m\geq 1$.
\bea
\left(\begin{array}{rrr}2320m^2-2057m+456\;&\;-(2240m^2-2208m+547)\;&\;
-(80m^2-89m+24)\\-(1840m^2-1639m+365)\;&\;2880m^2-2656m+621\;&\;
-(1040m^2-937m+211)\\-(480m^2-418m+91)\;&\;-(640m^2-448m+74)\;&\;
1120m^2-1026m+235\end{array}\right)\nonumber
\eea
{\bf 13}. $n=40m-16$, $m\geq 1$.
\bea
\left(\begin{array}{rrr}2080m^2-1629m+319\;&\;-(960m^2-1024m+259)\;&\;
-(1120m^2-925m+191)\\-(800m^2-635m+126)\;&\;1600m^2-1280m+261\;&\;
-(800m^2-645m+130)\\-(1280m^2-994m+193)\;&\;-(640m^2-256m+2)\;&\;
1920m^2-1570m+321\end{array}\right)\nonumber
\eea
{\bf 14}. $n=40m-14$, $m\geq 2$; $\qquad n\neq 26$.
\bea
\left(\begin{array}{rrr}2960m^2-2015m+343\;&\;-(320m^2-672m+197)\;&\;
-(2640m^2-1903m+343)\\-(880m^2-605m+104)\;&\;960m^2-736m+143\;&\;-(80m^2-51m+8)
\\-(2080m^2-1410m+239)\;&\;-(640m^2-64m-54)\;&\;2720m^2-1954m+351\end{array}
\right)\nonumber
\eea
{\bf 15}. $n=40m-12$, $m\geq 1$.
\bea
\left(\begin{array}{rrr}2240m^2-1308m+191\;&\;-4(320m^2-256m+49)\;&\;
-(960m^2-604m+95)\\-(1920m^2-1129m+166)\;&\;2240m^2-1472m+247\;&\;
-(320m^2-183m+26)\\-(320m^2-179m+25)\;&\;-(960m^2-448m+51)\;&\;1280m^2-787m+121
\end{array}\right)\nonumber
\eea
{\bf 16}. $n=40m-10$, $m\geq 2$; $\qquad n\neq 30$.
\bea
\left(\begin{array}{rrr}2160m^2-1039m+125\;&\;-(320m^2-480m+101)\;&\;
-(1840m^2-959m+125)\\-(640m^2-316m+39)\;&\;1280m^2-640m+84\;&\;-(640m^2-324m+41)
\\-(1520m^2-723m+86)\;&\;-(960m^2-160m-17)\;&\;2480m^2-1283m+166\end{array}
\right)\nonumber
\eea
{\bf 18}. $n=40m-8$, $m\geq 1$.
\bea
\left(\begin{array}{rrr}1760m^2-671m+64\;&\;-(320m^2-384m+65)\;&\;
-(1440m^2-607m+64)\\-(800m^2-315m+31)\;&\;1600m^2-640m+69\;&\;-(800m^2-325m+33)
\\-(960m^2-356m+33)\;&\;-(1280m^2-256m+4)\;&\;2240m^2-932m+97\end{array}
\right)\nonumber
\eea
{\bf 19}. $n=40m-6$, $m\geq 1$.
\bea
\left(\begin{array}{rrr}1360m^2-383m+27\;&\;-(320m^2-288m+37)\;&\;
-(1040m^2-335m+27)\\-(160m^2-58m+5)\;&\;1920m^2-448m+30\;&\;-(1760m^2-550m+43)
\\-(1200m^2-325m+22)\;&\;-(1600m^2-160m-7)\;&\;2800m^2-885m+70\end{array}
\right)\nonumber
\eea
{\bf 20}. $n=40m-4$, $m\geq 1$.
\bea
\left(\begin{array}{rrr}960m^2-175m+8\;&\;-(320m^2-192m+17)\;&\;-(640m^2-143m+8)
\\-(640m^2-135m+7)\;&\;2880m^2-448m+25\;&\;-(2240m^2-473m+25)\\-(320m^2-40m+1)
\;&\;-(2560m^2-256m+8)\;&\;2880m^2-616m+33\end{array}\right)\nonumber
\eea
{\bf 21}. $n=40m-2$, $m\geq 1$.
\bea
\left(\begin{array}{rrr}560m^2-47m+1\;&\;-(320m^2-96m+5)\;&\;-(240m^2-31m+1)\\
-(400m^2-65m+2)\;&\;4800m^2-160m+11\;&\;-(4400m^2-495m+14)\\-(160m^2+18m-1)\;&\;
-(4480m^2-64m+6)\;&\;4640m^2-526m+15\end{array}\right).\nonumber
\eea
\subsection{Numerical semigroups $R_n^4$, $\;n=1\;(\bmod\;2)$}\label{l42}
{\bf 1}. $n=20m+1$, $m\geq 1$.
\bea
\left(\begin{array}{rrr}230m^2+30m+1\;&\;-32m(5m-1)\;&\;-2m(35m+1)\\
-10m(15m+1)\;&\;16(50m^2+10m+1)\;&\;-(650m^2+50m+1)\\-(80m^2+20m+1)\;&\;
-16(40m^2+12m+1)\;&\;720m^2+52m+1\end{array}\right)\nonumber
\eea
{\bf 2}. $n=20m+3$, $m\geq 1$; $\qquad n\neq 3$.
\bea
\left(\begin{array}{rrr}530m^2+178m+15\;&\;-16(10m^2-6m-1)\;&\;-(370m^2+94m+6)\\
-(60m^2+16m+1)\;&\;16(20m^2+8m+1)\;&\;-(260m^2+72m+5)\\-(470m^2+162m+14)\;&\;
-32(5m^2+7m+1)\;&\;630m^2+166m+11\end{array}\right)\nonumber
\eea
{\bf 3}. $n=20m+5$, $m\geq 1$; $\qquad n\neq 5$.
\bea
\left(\begin{array}{rrr}490m^2+258m+34\;&\;-16(30m^2+10m+1)\;&\;-(10m^2-2m-1)\\
-(320m^2+164m+21)\;&\;16(40m^2+20m+3)\;&\;-(320m^2+156m+19)\\-(170m^2+94m+13)
\;&\;-32(5m^2+5m+1)\;&\;330m^2+154m+18\end{array}\right)\nonumber
\eea
{\bf 4}. $n=20m+7$, $m\geq 2$; $\qquad n\neq 7,27$.
\bea
\left(\begin{array}{rrr}2(565m^2+417m+77)\;&\;-32(5m^2-7m-3)\;&\;
-(970m^2+638m+105)\\-(130m^2+94m+17)\;&\;16(10m^2+6m+1)\;&\;-(30m^2+22m+4)\\
-(1000m^2+740m+137)\;&\;-16(20m+7)\;&\;1000m^2+660m+109\end{array}\right)
\nonumber
\eea
{\bf 5}. $n=20m+9$, $m\geq 1$; $\qquad n\neq 9$.
\bea
\left(\begin{array}{rrr}430m^2+402m+94\;&\;-16(10m^2+2m-1)\;&\;-(270m^2+230m+49)
\\-(290m^2+266m+61)\;&\;32(15m^2+13m+3)\;&\;-(190m^2+170m+38)\\
-(140m^2+136m+33)\;&\;-16(20m^2+24m+7)\;&\;460m^2+400m+87\end{array}\right)
\nonumber
\eea
{\bf 6}. $n=20m-9$, $m\geq 1$.
\bea
\left(\begin{array}{rrr}460m^2-400m+87\;&\;-16(20m^2-24m+7)\;&\;
-(140m^2-136m+33)\\-(190m^2-170m+38)\;&\;32(15m^2-13m+3)\;&\;-(290m^2-266m+61)\\
-(270m^2-230m+49)\;&\;-16(10m^2-2m-1)\;&\;430m^2-402m+94\end{array}\right)
\nonumber
\eea
{\bf 7}. $n=20m-7$, $m\geq 3$; $\qquad n\neq 13,33$.
\bea
\left(\begin{array}{rrr}1030m^2-682m+113\;&\;-32(5m^2-13m+4)\;&\;
-(870m^2-646m+120)\\-(30m^2-22m+4)\;&\;16(10m^2-6m+1)\;&\;-(130m^2-94m+17)\\
-(1000 m^2-660m+109)\;&\;-16(20 m-7)\;&\;1000m^2-740m+137\end{array}\right)
\nonumber
\eea
{\bf 8}. $n=20m-5$, $m\geq 1$.
\bea
\left(\begin{array}{rrr}330m^2-154m+18\;&\;-32(5m^2-5m+1)\;&\;-(170m^2-94m+13)\\
-(320m^2-156m+19)\;&\;16(40m^2-20m+3)\;&\;-(320m^2-164m+21)\\-(10m^2+2m-1)\;&\;
-16(30m^2-10m+1)\;&\;490m^2-258m+34\end{array}\right)\nonumber
\eea
{\bf 9}. $n=20m-3$, $m\geq 2$; $\qquad n\neq 17$.
\bea
\left(\begin{array}{rrr}630m^2-166m+11\;&\;-32(5m^2-7m+1)\;&\;-(470m^2-162m+14)
\\-(260m^2-72m+5)\;&\;16(20m^2-8m+1)\;&\;-(60m^2-16m+1)\\-(370m^2-94m+6)\;&\;
-16(10m^2+6m-1)\;&\;530m^2-178m+15\end{array}\right)\nonumber
\eea
{\bf 10}. $n=20m-1$, $m\geq 1$.
\bea
\left(\begin{array}{rrr}720m^2-52m+1\;&\;-16(40m^2-12m+1)\;&\;-(80m^2-20m+1)\\
-(650m^2-50m+1)\;&\;16(50m^2-10m+1)\;&\;-10m(15m-1)\\-2m(35m-1)\;&\;-32m(5m+1)
\;&\;230m^2-30m+1\end{array}\right).
\nonumber
\eea
\subsection{Exceptional semigroups $R_n^4$}\label{l43}
There exist 18 exceptional semigroups $R_n^4$ (3 symmetric $R_3^4$, $R_5^4$,
$R_7^4$ and 15 nonsymmetric) with minimal relations which do not obey the 
matrix Reps presented in sections \ref{l41}, \ref{l42}. We give them below.
\bea
&&R_3^4\equiv \la 2^4,3^4\ra,\;\;\left.\begin{array}{l}G=560\\F=1199
\end{array}\right.,\hspace{2cm}
R_7^4:\left(\begin{array}{rrr}256\;&\;0\;&\;-81\\-17\;&\;16\;&\;-4\\-256
\;&\;0\;&\;81\end{array}\right),\;\;\left.\begin{array}{l}G=181200,\\F=362399
\end{array}\right.\nonumber\\
&&R_5^4:\left(\begin{array}{rrr}81\;&\;0\;&\;-16\\-34\;&\;16\;&\;-1\\
-81\;&\;0\;&\;16\end{array}\right),\;\;\left.\begin{array}{l}G=14280\\
F=28559\end{array}\right.,\;\;
R_6^4:\left(\begin{array}{rrr}113\;&\;-23\;&\;-17\\-43\;&\;30\;&\;-5\\
-70\;&\;-7\;&\;22\end{array}\right),\quad\left.\begin{array}{l}G=41713,\\
F=78308.\end{array}\right.\nonumber\\
&&R_9^4:\left(\begin{array}{rrr}155\;&\;-80\;&\;-11\\-61\;&\;96\;&\;-38\\
-94\;&\;-16\;&\;49\end{array}\right),\quad\left.\begin{array}{l}G=502480,\\
F=994223.\end{array}\right.\nonumber\\
&&R_{10}^4:\left(\begin{array}{rrr}207\;&\;-67\;&\;-47\\-41\;&\;84\;&\;-39\\
-166\;&\;-17\;&\;86\end{array}\right),\quad\left.\begin{array}{l}G=965342,\\
F=1897924.\end{array}\right.\nonumber\\
&&R_{13}^4:\left(\begin{array}{rrr}485\;&\;-32\;&\;-238\\-12\;&\;80\;&\;-53\\
-473\;&\;-48\;&\;291\end{array}\right),\quad\left.\begin{array}{l}G=6071192,\\
F=12005295.\end{array}\right.\nonumber\\
&&R_{14}^4:\left(\begin{array}{rrr}359\;&\;-89\;&\;-135\\-8\;&\;143\;&\;-104\\
-351\;&\;-54\;&\;239\end{array}\right),\quad\left.\begin{array}{l}G=7729559,\\
F=15400797.\end{array}\right.\nonumber\\
&&R_{17}^4:\left(\begin{array}{rrr}668\;&\;-176\;&\;-277\\-193\;&\;208\;&\;-45
\\-475\;&\;-32\;&\;322\end{array}\right),\quad\left.\begin{array}{l}G=24979344,
\\F=48247935.\end{array}\right.\nonumber\\
&&R_{20}^4:\left(\begin{array}{rrr}1243\;&\;-85\;&\;-763\\-81\;&\;162\;&\;-79\\
-1162\;&\;-77\;&\;842\end{array}\right),\quad\left.\begin{array}{l}G=90813516,
\\F=176868200.\end{array}\right.\nonumber\\
&&R_{26}^4:\left(\begin{array}{rrr}1667\;&\;-212\;&\;-1043\\-379\;&\;367\;&\;-37
\\-1288\;&\;-155\;&\;1080\end{array}\right),\quad\left.\begin{array}{l}
G=365363593,\\F=720624113.\end{array}\right.\nonumber\\
&&R_{27}^4:\left(\begin{array}{rrr}2359\;&\;-112\;&\;-1657\\-241\;&\;272\;&\;
-561\\-2118\;&\;-160\;&\;1713\end{array}\right),\quad\left.\begin{array}{l}
G=647256024,\\F=1230618127.\end{array}\right.\nonumber
\eea
\bea
&&R_{30}^4:\left(\begin{array}{rrr}1609\;&\;-665\;&\;-649\\-363\;&\;724\;&\;-357
\\-1246\;&\;-59\;&\;1006\end{array}\right),\quad\left.\begin{array}{l}
G=739585479,\\F=1465271324.\end{array}\right.\nonumber\\
&&R_{33}^4:\left(\begin{array}{rrr}2949\;&\;-400\;&\;-1959\\-80\;&\;464\;&\;-349
\\-2869\;&\;-64\;&\;2308\end{array}\right),\quad\left.\begin{array}{l}
G=1782545568,\\F=3555061055.\end{array}\right.\nonumber\\
&&R_{40}^4:\left(\begin{array}{rrr}8805\;&\;-4\;&\;-7205\\-161\;&\;321\;&\;-159
\\-8644\;&\;-317\;&\;7364\end{array}\right),\quad\left.\begin{array}{l}
G=10589583194,\\F=21173668803.\end{array}\right.\nonumber\\
&&R_{60}^4:\left(\begin{array}{rrr}10205\;&\;-1203\;&\;-7805\\-723\;&\;1442\;&\;
-717\\-9482\;&\;-239\;&\;8522\end{array}\right),\quad\left.\begin{array}{l}
G=67447447193,\\F=133546213800.\end{array}\right.\nonumber\\
&&R_{80}^4:\left(\begin{array}{rrr}33605\;&\;-2\;&\;-30405\\-642\;&\;1281\;&\; 
-638\\-32963\;&\;-1279\;&\;31043\end{array}\right),\quad\left.\begin{array}{l}
G=680612207996,\\F=1361182355203.\end{array}\right.\nonumber\\
&&R_{120}^4:\left(\begin{array}{rrr}75367\;&\;-1922\;&\;-68647\\-1443\;&\;2881
\;&\;-1437\\-73924\;&\;-959\;&\;70084\end{array}\right),\quad\left.\begin{array}
{l}G=7758023870871,\\F=15421051483202.\end{array}\right.\nonumber
\eea
\subsection{Duality of semigroups $R_{T_4m-k}^4$ and $R_{T_4m+k}^4$, 
$T_4=40$}\label{l44}
A careful observation of matrices of minimal relations for semigroups $R_{T_4m
\pm k}^4$, \\$T_4=40$, and their genera allows to prove two statements.
\begin{theorem}\label{the1}
Let two semigroups be given by their minimal relations (\ref{x9}),
\bea
R_{T_4m\pm k}^4:\;\left(\begin{array}{rrr}E_{11}^{\pm}(m)\;&\;-E_{12}^{\pm}(m)
\;&\;-E_{13}^{\pm}(m)\\-E_{21}^{\pm}(m)\;&\;E_{22}^{\pm}(m)\;&\;-E_{23}^{\pm}(m)
\\-E_{31}^{\pm}(m)\;&\;-E_{32}^{\pm}(m)\;&\;E_{33}^{\pm}(m)\end{array}\right),
\quad k\leq \frac{T_4}{2},\label{j1}
\eea
where $E_{ij}^-(m)$ and $E_{ij}^+(m)$ are given by polynomials
\bea
E_{ij}^-(m)=A_{ij}^-m^2-B_{ij}^-m+C_{ij}^-,\quad
E_{ij}^+(m)=A_{ij}^+m^2+B_{ij}^+m+C_{ij}^+.\label{j2}
\eea
Then the following duality relations hold
\bea 
&&A_{11}^+=A_{33}^-,\quad B_{11}^+=B_{33}^-,\quad C_{11}^+=C_{33}^-,\quad 
A_{22}^+=A_{22}^-,\quad B_{22}^+=B_{22}^-,\quad C_{22}^+=C_{22}^-,\nonumber\\
&&A_{33}^+=A_{11}^-,\quad B_{33}^+=B_{11}^-,\quad C_{33}^+=C_{11}^-,\quad 
A_{12}^+=A_{32}^-,\quad B_{12}^+=B_{32}^-,\quad C_{12}^+=C_{32}^-,\nonumber\\
&&A_{21}^+=A_{23}^-,\quad B_{21}^+=B_{23}^-,\quad C_{21}^+=C_{23}^-,\quad 
A_{31}^+=A_{13}^-,\quad B_{31}^+=B_{13}^-,\quad C_{31}^+=C_{13}^-,\nonumber\\
&&A_{13}^+=A_{31}^-,\quad B_{13}^+=B_{31}^-,\quad C_{13}^+=C_{31}^-,\quad 
A_{23}^+=A_{21}^-,\quad B_{23}^+=B_{21}^-,\quad C_{23}^+=C_{21}^-,\nonumber\\
&&A_{32}^+=A_{12}^-,\quad B_{32}^+=B_{12}^-,\quad 
C_{32}^+=C_{12}^-.\nonumber  
\eea
\end{theorem}
\begin{proof}
According to formulas in sections \ref{l41}, \ref{l42} consider the polynomial 
Rep of minimal relations for numerical semigroups $R_{T_4m-k}^4$,
\bea
\left.\begin{array}{lclll}
E_{11}^-(m)(T_4m-k-1)^4&\;=\;&E_{12}^-(m)(T_4m-k)^4&\;+\;&E_{13}^-(m)
(T_4m-k+1)^4,\\
E_{22}^-(m)(T_4m-k)^4&\;=\;&E_{21}^-(m)(T_4m-k-1)^4&\;+\;&E_{23}^-(m)
(T_4m-k+1)^4,\\
E_{33}^-(m)(T_4m-k+1)^4&\;=\;&E_{31}^-(m)(T_4m-k-1)^4&\;+\;&E_{32}^-(m)
(T_4m-k)^4,\end{array}\right.\qquad\label{j3}
\eea
and $R_{40m+k}^4$,
\bea
\left.\begin{array}{lclll}
E_{11}^+(m)(T_4m+k-1)^4&\;=\;&E_{12}^+(m)(T_4m+k)^4&\;+\;&E_{13}^+(m)
(T_4m+k+1)^4,\\
E_{22}^+(m)(T_4m+k)^4&\;=\;&E_{21}^+(m)(T_4m+k-1)^4&\;+\;&E_{23}^+(m)
(T_4m+k+1)^4,\\
E_{33}^+(m)(T_4m+k+1)^4&\;=\;&E_{31}^+(m)(T_4m+k-1)^4&\;+\;&E_{32}^+(m)
(T_4m+k)^4,\end{array}\right.\qquad\label{j4}
\eea
where $E_{ij}^-(m)$ and $E_{ij}^+(m)$ are defined in (\ref{j2}). Replacing $m
\to -m$ in (\ref{j3}) we get
\bea
\left.\begin{array}{lclll}
E_{11}^-(-m)(T_4m+k+1)^4&\;=\;&E_{12}^-(-m)(T_4m+k)^4&\;+\;&E_{13}^-(-m)
(T_4m+k-1)^4,\\
E_{22}^-(-m)(T_4m+k)^4&\;=\;&E_{21}^-(-m)(T_4m+k+1)^4&\;+\;&E_{23}^-(-m)
(T_4m+k-1)^4,\\
E_{33}^-(-m)(T_4m+k-1)^4&\;=\;&E_{31}^-(-m)(T_4m+k+1)^4&\;+\;&E_{32}^-(-m)
(T_4m+k)^4.\end{array}\right.\nonumber
\eea
Compare the last three Eqs. with (\ref{j4}). Both Reps coincide for arbitrary 
$m$ iff
\bea
&&E_{11}^+(m)=E_{33}^-(-m),\qquad E_{12}^+(m)=E_{32}^-(-m),\qquad E_{13}^+(m)=
E_{31}^-(-m),\nonumber\\
&&E_{22}^+(m)=E_{22}^-(-m),\qquad E_{21}^+(m)=E_{23}^-(-m),\qquad E_{23}^+(m)=
E_{21}^-(-m),\nonumber\\
&&E_{33}^+(m)=E_{11}^-(-m),\qquad E_{31}^+(m)=E_{13}^-(-m),\qquad E_{32}^+(m)=
E_{12}^-(-m).\quad\label{j5}
\eea
Substituting (\ref{j2}) into (\ref{j5}) we arrive at the proof of Theorem.
\end{proof}
The next Theorem is motivated by systematic calculation on genera of semigroups
$R_{T_4m+k}^4$ and $R_{T_4m-k}^4$. We illustrate it by the following example,
\bea
G\left(R_{20m+9}^4\right)&=&8(7766000 m^6+21049600 m^5 +23809780 m^4 
+14385560 m^3 +\nonumber\\
&&\hspace{4.6cm}4895936 m^2+889781 m +67446),\nonumber\\
G\left(R_{20m-9}^4\right)&=&8(7766000 m^6-21049600 m^5 +23809780 m^4 
-14385560 m^3 +\nonumber\\
&&\hspace{4.6cm}4895936 m^2-889781 m +67446).\nonumber
\eea
\begin{theorem}\label{the2}
Let two numerical semigroups $R_{T_4m\mp k}^4$ be given by their minimal 
relations (\ref{j1}) and let their genera $G^{\pm}(m)$ are given by formulas,
\bea
R_{T_4m-k}^4:\;\;G^-(m)=\sum_{r=0}^6g_r^-m^r\qquad R_{T_4m+k}^4:\;\;G^+(m)=
\sum_{r=0}^6g_r^+m^r.\label{j6}
\eea
Then $g_{2r}^-=g_{2r}^+$ and $g_{2r+1}^-=-g_{2r+1}^+$.
\end{theorem}
\begin{proof}
Write formulas (\ref{x10}) of $D_k$, $0\leq k\leq 3$, for two numerical 
semigroups: $R_{T_4m- k}^4$
\bea
&&D_0^-(m)=E_{11}^-(m)\;E_{22}^-(m)\;E_{33}^-(m),\qquad
D_1^-(m)=E_{12}^-(m)\;E_{23}^-(m)\;E_{31}^-(m),\nonumber\\
&&D_2^-(m)=E_{13}^-(m)\;E_{32}^-(m)\;E_{21}^-(m),\nonumber\\
&&D_3^-(m)=(T_4m-k-1)^4+(T_4m-k)^4+(T_4m-k+1)^4,\nonumber
\eea
and $R_{40m+k}^4$,
\bea
&&D_0^+(m)=E_{11}^+(m)\;E_{22}^+(m)\;E_{33}^+(m),\qquad
D_1^+(m)=E_{12}^+(m)\;E_{23}^+(m)\;E_{31}^+(m),\nonumber\\
&&D_2^+(m)=E_{13}^+(m)\;E_{32}^+(m)\;E_{21}^+(m),\nonumber\\
&&D_3^+(m)=(T_4m+k-1)^4+(T_4m+k)^4+(T_4m+k+1)^4.\nonumber
\eea
Combining the above formulas with (\ref{j5}) we obtain
\bea
D_0^-(m)\!=\!D_0^+(-m),\;\;D_1^-(m)\!=\!D_2^+(-m),\;\;
D_2^-(m)\!=\!D_1^+(-m),\;\;D_3^-(m)\!=\!D_3^+(-m).\nonumber
\eea
that together with genus definition in (\ref{x9}) gives $G^-(m)=G^+(-m)$. 
Combining the last equality with (\ref{j6}) we arrive at the proof of Theorem.
\end{proof}
\section{Concluding remarks}\label{l5}
In the present section we state a conjecture and put a question devoted to 
numerical semigroups $R_n^k$, $n>3,k\geq 5$, where an appearence of symmetric 
semigroups $R_n^k$ seems very rare. Numerical calculations give only two 
semigroups $R_5^{11}$ and $R_5^{13}$ among others $R_{2p+1}^k$, $2\leq p\leq 
50$, $5\leq k\leq 10^3$.
\begin{conjecture}\label{con1}
Let a numerical semigroup $R_n^k$, $n=T_km+j$, be given by their minimal 
relations on residue class of $n$ modulo $T_k$,
\bea
R_{T_km+j}^k:\;\left(\begin{array}{rrr}E_{11}^{(k)}(m)\;&\;-E_{12}^{(k)}(m)\;&\;
-E_{13}^{(k)}(m)\\-E_{21}^{(k)}(m)\;&\;E_{22}^{(k)}(m)\;&\;-E_{23}^{(k)}(m)\\
-E_{31}^{(k)}(m)\;&\;-E_{32}^{(k)}(m)\;&\;E_{33}^{(k)}(m)\end{array}\right),
\quad j\leq \frac{T_k}{2}.\label{i1}
\eea

If $k=2q$, then polynomial $E_{ij}^{(k)}(m)$ reads,
\bea
&&E_{ij}^{(2q)}(m)=A_{ij}m^q+B_{ij}m^{q-1}+\ldots +C_{ij}m+D_{ij},\;\;1\leq i,j
\leq 3,\label{i2}
\eea
and the Frobenius number and genus have the asymptotics: $F(n),G(n)={\cal O}
\left(n^{3q}\right)$.

If $k=2q+1$, then the matrix elements with $(i,j)=(1,1),(1,2),(2,1),(2,2)$ 
are given by
\bea
E_{ij}^{(2q+1)}(m)=K_{ij}m^{q+1}+I_{ij}m^q+\ldots +J_{ij}m+H_{ij},\label{i3}
\eea
while the matrix elements with $(i,j)=(1,3),(2,3),(3,1),(3,2),(3,3)$ read
\bea
E_{ij}^{(2q+1)}(m)=M_{ij}m^q+N_{ij}m^{q-1}+\ldots +P_{ij}m+S_{ij},\label{i4}
\eea
and the Frobenius number and genus have the asymptotics: $F(n),G(n)={\cal O}
\left(n^{3q+2}\right)$.
\end{conjecture}
\begin{question}\label{que1}
Keeping in mind $T_2=4,\;T_3=18,\;T_4=40$, find $T_k$ for $k\geq 5$.
\end{question}
\section*{Acknowledgement}
The research was supported by the Kamea Fellowship.
\appendix   
\section{Proof of Propositions}\label{app1}
{\bf Proof of Proposition \ref{pro2}}. 
Semigroup $R_3^3$ is symmetric due to requirement (\ref{x4}a). Find more $n$ 
which satisfy (\ref{x4}a),
\bea
(n+1)^3=e_1(n-1)^3+e_2n^3,\quad e_1,e_2\in{\mathbb N},\quad n>3.\nonumber
\eea
Simplifying the last equality we obtain the Diophantine Eq.
\bea
&&(e_1+e_2-1)t^3+3(2e_1+3e_2-4)t^2+3(4e_1+9e_2-16)t+\label{z1}\\
&&\hspace{4cm}8e_1+27e_2-64=0,\qquad t=n-3.\nonumber
\eea
Decompose the whole integer lattice ${\mathbb Z}_2^+:=\{e_1,e_2\;|\;e_1,e_2\geq 
1\}$ in different sets,   
\bea
&&{\mathbb Z}_2^+=\bigcup_{j=1}^5{\mathbb E}_j,\quad{\mathbb E}_1=\{e_1,e_2\;|
\;e_1\geq 5;\;e_2=1\},\quad{\mathbb E}_2=\{e_1,e_2\;|\;e_1\geq 2;\;e_2=2\},
\nonumber\\
&&{\mathbb E}_3=\{e_1,e_2\;|\;e_1\geq 1;\;e_2\geq 3\},\quad{\mathbb E}_4=\{e_1,
e_2\;|\;1\leq e_1\leq 4;\;e_2=1\},\nonumber\\
&&{\mathbb E}_5=\{e_1=e_2=1\}.\nonumber
\eea
If $(e_1,e_2)\in{\mathbb E}_j$, $1\leq j\leq 3$, then the sequence of
coefficients in Eq. (\ref{x8}) has no changes of signs and therefore, by  
Descartes' rule of signs, Eq. (\ref{x8}) has no positive solutions in $t$. If 
$(e_1,e_2)\in{\mathbb E}_j$, $j=4,5$, then a straightforward numerical
verification shows that neither of 5 qubic Eqs. (\ref{z1}) has integer 
positive 
solution in $t$.

Consider an alternative way to symmetrize $R_n^3$ by providing condition
(\ref{x4}b), which may occur only when $n=2q+1$ and results in the Diophantine
Eq. in $c_1,c_2\in{\mathbb N}$, $q>1$, $(2q+1)^3=c_1q^3+c_2(q+1)^3$, i.e., 
\bea
(c_1+c_2-8)q^3+3(c_2-4)q^2+3(c_2-2)q+c_2-1=0.\label{z2}
\eea
Substituting $q=p+1$, $p>0$, into (\ref{z2}) we obtain the cubic Diophantine
Eq. in $p$,
\bea
(c_1+c_2-8)p^3+3(c_1+2c_2-12)p^2+3(c_1+4c_2-18)p+c_1+8c_2-27=0,\qquad\label{z3}
\eea
which has no positive integer solutions $p$. Indeed, to prove this statement,
we make use of Descartes' rule of signs for integer coefficients in Eq. 
(\ref{z3}). For this purpose decompose the whole integer lattice ${\mathbb 
Z}
_2^+:=\{c_1,c_2\;|\;c_1,c_2\geq 1\}$ as follows,
\bea
{\mathbb Z}_2^+={\mathbb C}\cup{\overline{\mathbb C}},\qquad{\mathbb C}=
\bigcup_{j=1}^7{\mathbb C}_j,\qquad {\overline{\mathbb C}}=\bigcup_{j=1}^6
{\overline{\mathbb C}_j},\quad\mbox{where}\nonumber
\eea
\bea
&&{\mathbb C}_1=\{c_1,c_2\;|\;1\leq c_1\leq 7,\;c_1\geq 19;\;c_2=1\},\nonumber\\
&&{\mathbb C}_2=\{c_1,c_2\;|\;1\leq c_1\leq 6,\;c_1\geq 11;\;c_2=2\},\nonumber\\
&&{\mathbb C}_3=\{c_1,c_2\;|\;1\leq c_1\leq 3,\;c_1\geq 6;\;c_2=3\},\qquad
{\mathbb C}_4=\{c_1,c_2\;|\;c_1\geq 4;\;c_2=4\},\nonumber\\
&&{\mathbb C}_5\!=\!\{c_1,c_2\;|\;c_1\geq 3;\;c_2=5\},\qquad{\mathbb C}_6\!=\!  
\{c_1,c_2\;|\;c_1\geq 2;\;c_2=6\},\nonumber\\
&&{\mathbb C}_7\!=\!\{c_1,c_2\;|\;c_1\geq 1;\;c_2\geq 7\},\nonumber\\
&&{\overline{\mathbb C}_1}=\{c_1,c_2\;|\;8\leq c_1\leq 18;\;c_2=1\},\qquad   
{\overline{\mathbb C}_2}=\{c_1,c_2\;|\;7\leq c_1\leq 10;\;c_2=2\},\nonumber\\
&&{\overline{\mathbb C}_3}=\{c_1,c_2\;|\;4\leq c_1\leq 5;\;c_2=3\},\qquad
{\overline{\mathbb C}_4}=\{c_1,c_2\;|\;1\leq c_1\leq 3;\;c_2=4\},\nonumber\\
&&{\overline{\mathbb C}_5}=\{c_1,c_2\;|\;1\leq c_1\leq 2;\;c_2=5\},\qquad
{\overline{\mathbb C}_6}=\{c_1=1;\;c_2=6\}.\nonumber
\eea
If $(c_1,c_2)\in{\mathbb C}$ then the sequence of coefficients in Eq. 
(\ref{z3}) has no changes of signs and therefore, by Descartes' rule of signs, 
Eq. (\ref{z3}) has no positive solutions in $p$. Regarding the rest of the
cases, when $(c_1,c_2)\in{\overline{\mathbb C}}$, a straightforward numerical
verification shows that neither of 23 cubic Eqs. (\ref{z3}) has integer positive
solution in $p$.$\;\;\;\;\;\;\Box$
\vspace{.5cm}

{\bf Proof of Proposition \ref{pro3}}. Semigroup $R_3^4$ is symmetric 
due to requirement (\ref{x4}a). Find more $n$ which satisfy (\ref{x4}a),
\bea
(n+1)^4=f_1(n-1)^4+f_2n^4,\quad f_1,f_2\in{\mathbb N},\quad n>4.\nonumber
\eea
Simplify the last equality and obtain the Diophantine Eq.
\bea
&&(f_1+f_2-1)u^4+4(2f_1+3f_2-4)u^3+6(4f_1+9f_2-16)u^2+\label{z4}\\
&&\hspace{1cm} 4(8f_1+27f_2-64)u+16f_1+81f_2-256=0,\qquad u=n-3.\nonumber
\eea
Decompose the whole integer lattice ${\mathbb Z}_2^+:=\{f_1,f_2\;|\;f_1,f_2
\geq 1\}$ in different sets, ${\mathbb Z}_2^+=\bigcup_{j=1}^5{\mathbb F}_j$,
 where
\bea
&&{\mathbb F}_1=\{f_1,f_2\;|\;f_1\geq 11;\;f_2=1\},\qquad
{\mathbb F}_2=\{f_1,f_2\;|\;f_1\geq 6;\;f_2=2\},\nonumber\\
&&{\mathbb F}_3\!=\!\{f_1,f_2\;|\;f_1\geq 1;\;f_2\geq 3\},\qquad
{\mathbb F}_4=\{f_1,f_2\;|\;1\leq f_1\leq 10;\;f_2=1\},\nonumber\\
&&{\mathbb F}_5=\{f_1,f_2\;|\;1\leq f_1\leq 5;\;f_2=2\}.\nonumber
\eea
If $(f_1,f_2)\in{\mathbb F}_j$, $1\leq j\leq 3$, then the sequence of
coefficients in Eq. (\ref{z4}) has no changes of signs and therefore, by
Descartes' rule of signs, Eq. (\ref{z4}) has no positive solutions in $u$. In 
the rest of the cases, when $(f_1,f_2)\in{\mathbb F}_j$, $j=4,5$, a
straightforward numerical verification shows that neither of 15 quartic Eqs. 
(\ref{z4}) has integer positive solution in $u$.

Consider another way to find symmetric semigroups $R_n^4$ by providing 
condition (\ref{x4}b), which may occur only when $n=2q+1$ and results in 
the Diophantine Eq. in $h_1,h_2\in{\mathbb N}$, $q>1$,
\bea
(2q+1)^4=h_1q^4+h_2(q+1)^4.\label{z5}
\eea
Equation (\ref{z5}) has solutions $q=2,h_1=34,h_2=1$ and $q=3,h_1=17,h_2=4$, 
which correspond to symmetric semigroups $R_5^4$ and $R_7^4$, respectively. We 
show that Eq. (\ref{z5}) has no more positive integer solutions. Denote $v=q-2$,
$v>0$, and represent (\ref{z5}) as follows,
\bea
(h_1+h_2-16)v^4+4(2h_1+3h_2-40)v^3+6(4h_1+9h_2-100)v^2+\nonumber\\
4(8h_1+27h_2-250)v+16h_1+81h_2-625=0.\label{z6}
\eea
Decompose the whole integer lattice ${\mathbb Z}_2^+:=\{h_1,h_2\;|\;h_1,h_2
\geq 1\}$ in different sets,
\bea
{\mathbb Z}_2^+={\mathbb H}\cup{\overline{\mathbb H}},\qquad 
{\mathbb H}=\bigcup_{j=1}^{15}{\mathbb H}_j,\qquad {\overline{\mathbb H}}=
\bigcup_{j=0}^{14}{\overline{\mathbb H}_j},\quad\mbox{where}\nonumber
\eea
\bea
&&{\mathbb H}_1\!=\!\{h_1,h_2\;|\;1\leq h_1\leq 15,\;h_1\geq 34;\;h_2\!=\!1\},
\nonumber\\
&&{\mathbb H}_2\!=\!\{h_1,h_2\;|\;1\leq h_1\leq 14,\;h_1\geq 29;\;h_2\!=\!2\},
\nonumber\\
&&{\mathbb H}_3\!=\!\{h_1,h_2\;|\;1\leq h_1\leq 13,\;h_1\geq 24;\;h_2\!=\!3\},
\nonumber\\
&&{\mathbb H}_4\!=\!\{h_1,h_2\;|\;1\leq h_1\leq 12,\;h_1\geq 19;\;h_2\!=\!4\},
\nonumber\\
&&{\mathbb H}_5\!=\!\{h_1,h_2\;|\;1\leq h_1\leq 11,\;h_1\geq 15;\;h_2\!=\!5\},
\nonumber\\
&&{\mathbb H}_6\!=\!\{h_1,h_2\;|\;1\leq h_1\leq 8,\;h_1\geq 12;\;h_2\!=\!6\},
\nonumber\\
&&{\mathbb H}_7=\{h_1,h_2\;|\;1\leq h_1\leq 3,\;h_1\geq 10;\;h_2=7\},\nonumber\\
&&{\mathbb H}_j=\{h_1,h_2\;|\;h_1\geq 16-j;\;h_2=j\},\;\;8\leq j\leq 14,
\nonumber\\
&&{\mathbb H}_{15}=\{h_1,h_2\;|\;h_1\geq 1;\;h_2\geq 15\},\nonumber\\
&&{\overline{\mathbb H}_1}=\{h_1,h_2\;|\;16\leq h_1\leq 33;\;h_2=1\},\qquad
{\overline{\mathbb H}_2}=\{h_1,h_2\;|\;15\leq h_1\leq 28;\;h_2=2\},\nonumber\\
&&{\overline{\mathbb H}_3}=\{h_1,h_2\;|\;14\leq h_1\leq 23;\;h_2=3\},\qquad
{\overline{\mathbb H}_4}=\{h_1,h_2\;|\;13\leq h_1\leq 18;\;h_2=4\},\nonumber\\
&&{\overline{\mathbb H}_5}=\{h_1,h_2\;|\;12\leq h_1\leq 14;\;h_2=5\},\qquad
{\overline{\mathbb H}_6}=\{h_1,h_2\;|\;9\leq h_1\leq 11;\;h_2=6\},\nonumber\\
&&{\overline{\mathbb H}_7}=\{h_1,h_2\;|\;4\leq h_1\leq 9;\;h_2=7\},\nonumber\\
&&{\overline{\mathbb H}_j}=\{h_1,h_2\;|\;1\leq h_1\leq 15-j;\;h_2=j\},\quad
8\leq j\leq 14.\nonumber
\eea
If $(h_1,h_2)\in{\mathbb H}$ then the sequence of coefficients in Eq. 
(\ref{z6}) has no changes of signs and therefore, by Descartes' rule of signs,
Eq. (\ref{z6}) has no positive solutions in $p$. Regarding the rest of the 
cases, when $(h_1,h_2)\in{\overline{\mathbb H}}$, a straightforward numerical 
verification shows that neither of 88 quartic Eqs. (\ref{z6}) has integer 
positive solution in $v$.$\;\;\;\;\;\;\Box$


\begin{thebibliography}{99}
\bibitem{bra42}  A. Brauer, On a problem of partitions, {\it Am. J. Math.},
                 {\bf 64} (1942) 299-312
\bibitem{fel06}  L. Fel, Frobenius problem for semigroups ${\sf S}\left(
                 d_1,d_2,d_3\right)$, {\it Funct. Analysis and Other Math.}, 
                 {\bf 1} (2006) \# 2, 119-157
\bibitem{f09}    L. Fel, Symmetric semigroups generated by Fibonacci and Lucas 
                 triples, {\it Integers}, {\bf 9} (2009) 106-116
\bibitem{f11}    L. Fel, Duality relation for the Hilbert series of almost 
                 symmetric numerical semigroups, {\it Israel J. Math}, {\bf 185}
                 (2011) 413-444
\bibitem{herz70} J. Herzog, Generators and relations of Abelian semigroups
                 and semigroup rings,\\ {\it Manuscripta Math.}, {\bf 3} (1970) 
                 175-193
\bibitem{john60} S. Johnson, A linear Diophantine problem, {\it Canad. J. Math.}
                 , {\bf 12} (1960) 390-398
\bibitem{kraf85} J. Kraft, Singularity of monomial curves in ${\mathbb A}^3$ 
                 and Gorenstein monomial curves in ${\mathbb A}^4$, {\it Canad. 
                 J. Math.}, {\bf 37} (1985) 872-892
\bibitem{le15}   M. Lepilov, J. O'Rourke, I. Swanson, Frobenius numbers of
                 numerical semigroups generated by three consecutive squares
                 or cubes, {\it Semigroup Forum}, {\bf 91} (2015) 238-259 
\bibitem{mrr07}  J. Marin, J. Ramirez Alfonsin, M. Revuelta, On the Frobenius 
                 number of Fibonacci numerical semigroups, {\it Integers},
                 {\bf 7} (2007) \# A14
\bibitem{op08}   D. Ong, V. Ponomarenko, Frobenius number of geometric
                 sequences, {\it Integers}, {\bf 8} (2008) \# A33
\bibitem{rob56}  J. Roberts, Note on linear forms, {\it Proc. Am. Math. Soc.},
                 {\bf 7} (1956) 465-469
\bibitem{ro78}   J. R\"odseth, A linear Diophantine problem of Frobenius,
                 {\it J. Reine Angew. Math.}, {\bf 301} (1978) 171-178   
\bibitem{rr09}   J. Ramirez Alfonsin and J. R\"odseth, Numerical semigroups:  
                 Ap\'ery sets and Hilbert series, {\it Semigroup Forum}, 
                 {\bf 79} (2009) 323-340
\bibitem{Syl84}  J. Sylvester, Mathematical Questions with their solution, 
                 {\it Educational Times}, {\bf 41} (1884) 21
\bibitem{wata73} K. Watanabe, Examples of 1--dim Gorenstein Domains, {\it 
                 Nagoya Math. J.}, {\bf 49} (1973) 101-109
\end{thebibliography}
\end{document}